\documentclass{amsart}

\usepackage{epsfig,amsmath,xypic}
\usepackage[psamsfonts]{amssymb}
\usepackage[utf8]{inputenc}
\usepackage{pstricks,pst-node,pst-plot}
\usepackage{hyperref}
\usepackage{diagbox}

\newtheorem{theorem}{\bf Theorem}
\newtheorem{proposition}[theorem]{\bf Proposition}

\theoremstyle{remark}

\newtheorem{definition}{\bf Definition}

\def\and{{\quad\text{and}\quad}}

\def\Z{{\mathbb Z}}
\def\Q{{\mathbb Q}}
\def\Qbar{{\overline \Q}}
\def\R{{\mathbb R}}

\def\C{{\mathbb C}}

\def\mod#1{{\ ({\rm mod}\ #1)}}


\begin{author}[X.~Buff]{Xavier Buff}
\email{xavier.buff@math.univ-toulouse.fr}
\address{ %
  Institut de Math\'ematiques de Toulouse\\
 Universit\'e Paul Sabatier\\
  118, route de Narbonne \\
  31062 Toulouse Cedex \\
  France }
\end{author}

\begin{author}[S.~Koch]{Sarah Koch}\thanks{The research of the second author was supported in part by the NSF}
\email{kochsc@umich.edu}
\address{Department of Mathematics\\
530 Church Street\\
East Hall\\
University of Michigan\\
Ann Arbor MI 48109\\
United States }
\end{author}

\title{Totally real points in the Mandelbrot set}

\begin{document}

\begin{abstract}
Recently, Noytaptim and Petsche proved that the only totally real parameters $c\in \Qbar$ for which $f_c(z):=z^2+c$ is postcritically finite are $0$, $-1$ and $-2$ \cite{clay}. In this note, we show that the only totally real parameters $c\in \Qbar$ for which $f_c$ has a parabolic cycle are $\frac14$, $-\frac34$, $-\frac54$ and $-\frac74$. 
\end{abstract}

\maketitle

\section*{Introduction}
\noindent Consider the family of quadratic polynomials $f_c:\C\to \C$ defined by 
\[f_c(z) := z^2+c,\quad c\in\C.\]
The Mandelbrot set $M$ 
is the set of parameters $c\in \C$ for which the orbit of the critical point $0$ under iteration of $f_c$ remains bounded:
\[M:=\bigl\{c\in \C~|~\forall n\geq 1,~f_c^{\circ n}(0) \in \overline D(0,2)\bigr\}.\]


\begin{definition}
A parameter $c\in \C$ is {\em postcritically finite} if the orbit of $0$ under iteration of $f_c$ is finite. 
\end{definition}

\begin{definition}
A parameter $c\in \C$ is {\em parabolic} if $f_c$ has a periodic cycle with multiplier a root of unity. 
\end{definition}

Postcritically finite parameters and parabolic parameters are algebraic numbers contained in $M$. More precisely, $c\in \C$ is a postcritically finite parameter if and only if $c$ is an algebraic integer whose Galois conjugates all belong to $M$ (see \cite{milnor2} and \cite{xavier}). In addition, if $c\in \C$ is a parabolic parameter, then $4c$ is an algebraic integer (see \cite{thierry}); moreover, $4c$ is an algebraic unit in $\overline \Z/2\overline \Z$ (see \cite[Remark 3.2]{milnor2}). 

\begin{definition}
An algebraic number $c\in\Qbar$  is {\em totally real} if its Galois conjugates are all in $\Qbar\cap \R$. 
\end{definition}

Recently, Noytaptim and Petsche \cite{clay} completely determined the totally real postcritically finite parameters.

\begin{proposition}[Noytaptim-Petsche]\label{prop:pcf}
The only totally real parameters  $c\in \Qbar$ for which $z\mapsto z^2+c$ is postcritically finite are $-2$, $-1$ and $0$. 
\end{proposition}

Their proof relies on the fact that the Galois conjugates of a postcritically finite parameter are also postcritically finite parameters, thus contained in $M$, and on the fact that $M\cap \R = [-2,\frac14]$ has small arithmetic capacity.  
In this note, we revisit their proof. We then determine the totally real parabolic parameters.

\begin{proposition}\label{prop:parabo}
The only totally real parameters $c\in \Qbar$ for which $z\mapsto z^2+c$ has a parabolic cycle are $\frac14$, $-\frac34$, $-\frac54$ and $-\frac74$. 
\end{proposition}

\noindent{\bf Acknowledgments.} We thank Valentin Huguin for useful discussions, and Curtis McMullen for introducing us to these questions. 

\section{Postcritically finite parameters}

\noindent We first revisit the proof of Noytaptim and Petsche \cite{clay}. 

\begin{proof}[Proof of Proposition \ref{prop:pcf}]
Assume that $c$ is a totally real postcritically finite parameter. Then, $c$ and all of its Galois conjugates are real postcritically finite parameters, thus lie in the interval $[-2,0]$. Indeed, 
\begin{itemize}
\item for $c\in (0,\frac14)$, the orbit of $0$ under iteration $f_c$ is infinite, converging to an attracting fixed point of $f_c$; 
\item for $c=\frac14$ the orbit of $0$ under iteration $f_c$ is infinite, converging to a parabolic fixed point of $f_c$ at $z=\frac12$; 
\item for $c\in (-\infty,-2)\cup (\frac14,+\infty)$ the orbit of $0$ under iteration $f_c$ is infinite, converging to $\infty$. 
\end{itemize}

Let $a$ be a solution of  $a+1/a = c$; that is, $a^2-ca+1=0$.
Then $a$ is an algebraic integer of modulus $1$ with nonpositive real part, and all of its Galois conjugates also have modulus $1$ and nonpositive real part. 

By Kronecker's theorem, $a$ is a root of unity. And since the Galois conjugates of $a$ all have nonpositive real part, the only possibilities are the following:
\begin{itemize}
\item $a = -1$, which is mapped to $c=-2$;
\item $a = {\rm e}^{\pm {\rm i}2\pi/3}$, which is mapped to $c=-1$; 
\item $a = \pm {\rm i}$, which is mapped to $c=0$. 
\end{itemize}
Therefore, the only postcritically finite parameters  that are totally real are $-2$, $-1$, and $0$.
\end{proof} 


\section{Parabolic parameters}

\noindent We now present the proof of Proposition \ref{prop:parabo}. Note that $c=\frac14$, $c=-\frac34$, $c=-\frac54$ and $c=-\frac74$ are indeed parabolic parameters. Indeed, 
\begin{itemize}
\item $f_{\frac14}$ has a fixed point with multiplier $1$ at $z=\frac12$; 
\item $f_{-\frac34}$ has a fixed point with multiplier $-1$ at $z=-\frac12$; 
\item $f_{-\frac54}$ has a cycle of period $2$ with multiplier $-1$ consisting of the two roots of $4z^2 + 4z - 1$; 
\item $f_{-\frac74}$ has a cycle of period $3$ with multiplier $1$ consisting of the three roots of $8z^3 + 4z^2 - 18z - 1$.
\end{itemize}

\begin{proof}[Proof of Proposition \ref{prop:parabo}]
Assume that $c$ is a totally real parabolic parameter. Then, the Galois conjugates of $c$ also are parabolic parameters. Either $c=\frac14$, $c=-\frac34$, $c=-\frac54$, or $c$ and all of its Galois conjugates lie in the interval $[-2,-\frac54)$. Indeed, a parabolic cycle must attract the orbit of $0$ under iteration of $f_c$. However, 
\begin{itemize}
\item for $c\in (-\frac34,\frac14)$, the orbit of $0$ under iteration $f_c$ converges to an attracting fixed point of $f_c$; 
\item for $c\in (-\frac54,-\frac34)$ the orbit of $0$ under iteration $f_c$ converges to an attracting cycle of period $2$ of $f_c$; 
\item for $c\in (-\infty,-2)\cup (\frac14,+\infty)$ the orbit of $0$ under iteration $f_c$ converges to $\infty$. 
\end{itemize}

Let us assume that $c\in [-2,-\frac54)$. Then, $b:=4c+6$ and all of its Galois conjugates lie in the interval $[-2,1)$. 
Let $a$ be a solution of  $a+1/a = b$; that is, $a^2-ba+1=0$.
Then $a$ is an algebraic integer of modulus $1$ with real part less than $\frac12$, and all of its Galois conjugates also have modulus $1$ and  real part less than $\frac12$.  

By Kronecker's theorem, $a$ is a root of unity. And since the Galois conjugates of $a$ all have real part less than $\frac12$, the only possibilities are the following:
\begin{itemize}
\item $a = -1$, $b=-2$ and $c = -2$; this is not a parabolic parameter; 
\item $a = {\rm e}^{\pm {\rm i}2\pi/3}$, $b=-1$ and $c = -\frac74$; this is indeed a parabolic parameter;
\item $a = \pm {\rm i}$, $b=0$ and $c = -\frac32$; in that case $4c = -6$ is not an algebraic unit in $\overline \Z/2\overline \Z$ and so, $c$ is not a parabolic parameter;
\item $a = {\rm e}^{\pm {\rm i}2\pi/5}$,  $b=2\cos(\frac{2\pi}{5})$ and  $c = \frac{\sqrt{5}-13}{8}$; in this case, $f_c$ has an attracting cycle of period $4$ and so, $c$ is not a parabolic parameter;
\item $a = {\rm e}^{\pm {\rm i}4\pi/5}$,  $b=2\cos(\frac{4\pi}{5})$ and $c =  \frac{-\sqrt{5}-13}{8}$; then the Galois conjugate $ \frac{\sqrt{5}-13}{8}$ is not a parabolic parameter and so, $c$ is not a parabolic parameter. 
\end{itemize}
This completes the proof of the proposition. 
\end{proof}

Remark: the following proof that $-\frac32$ is not a parabolic parameter was explained to us by Valentin Huguin. It follows from \cite{thierry} that for all $n\geq 1$, 
\[{\rm discriminant}\bigl(f_c^{\circ n}(z)-z,z\bigr) = P_n(4c)\quad \text{with}\quad P_n(b)\in \Z[b]\quad \text{and}\quad \pm P_n\text{ monic}.\]
As an example, 
\[P_1(b) = -b+1 ,\quad P_2(b) =(b-1)(b+3)^3 , \quad P_3(z)= (b-1 )(b+7)^3(b^2 + b + 7)^4,\]
and
\[P_4(z)=(b - 1)(b + 3)^3(b + 5)^6(b^3 + 9b^2 + 27b + 135)^4(b^2 - 2b + 5)^5.\]
Note that this yields an alternate proof that $c=\frac14$, $c=-\frac34$, $c=-\frac54$ and $c=-\frac74$ are parabolic parameters.
In addition, 
\[P_n(0) = {\rm discriminant}\bigl(z^{2^n}-z,z\bigr) \equiv 1\mod 2.\]
As a consequence
\[P_n(-6)\equiv 1\mod 2.\]
Thus, for all $n\geq 1$, the roots of $f_{-\frac32}^{\circ n}(z)-z$ are simple, which shows that $f_{-\frac32}$ has no parabolic cycle.

%
%
%
%
%

\bigskip

\end{document}